\documentclass[12pt, reqno]{amsart}
\usepackage{amssymb,latexsym,amsmath,amsfonts, amsthm}
\usepackage{latexsym}
\usepackage[numbers,sort&compress]{natbib}
\usepackage{color}
\usepackage[mathscr]{eucal}
\usepackage{comment}

\usepackage[linkcolor=blue, citecolor=black]{hyperref}
\hypersetup{colorlinks=true, urlcolor=blue}
\usepackage[nameinlink]{cleveref}
\usepackage{enumerate}

\numberwithin{equation}{section}

\theoremstyle{definition}
\newtheorem{definition}{Definition}[section]
\newtheorem{example}[definition]{Example}

\newtheorem{remark}[definition]{Remark}
\theoremstyle{plain}
\newtheorem{theorem}[definition]{Theorem}
\newtheorem{lemma}[definition]{Lemma}

\newtheorem{result}[definition]{Result}

\newtheorem{corollary}[definition]{Corollary}
\newtheorem{question}[definition]{Question}

\newtheorem{conjecture}[definition]{Conjecture}

\newcommand{\zbar}{\overline{z}}

\newcommand{\bdy}{\partial}
\newcommand{\Om}{\Omega}

\newcommand{\dsc}{\mathbb{D}}

\newcommand{\smoo}{\mathcal{C}}

\newcommand{\poly}{\mathscr{P}}

\newcommand{\rl}{{\sf Re}}
\newcommand{\imag}{{\sf Im}}

\newcommand\ba[1]{\overline{#1}}
\newcommand\hull[1]{\widehat{#1}}

\newcommand{\CC}{\mathbb{C}^2}
\newcommand{\cplx}{\mathbb{C}}

\newcommand{\rea}{\mathbb{R}}

\newcommand{\Gr}{{\sf Gr}}

\begin{document}
	\title[Compacts lying in certain Levi-flat hypersurfaces]{Polynomial convexity of compacts that lies in certain Levi-flat hypersurfaces in $\mathbb{C}^2$}
	\author{Sushil Gorai and Golam Mostafa Mondal}
	\address{Department of Mathematics and Statistics, Indian Institute of Science Education and Research Kolkata,
		Mohanpur -- 741 246}
	\email{sushil.gorai@iiserkol.ac.in, sushil.gorai@gmail.com}
	
	\address{Department of Mathematics, Indian Institute of Science Education and Research Pune,
		Pune -- 411 008}
	
	\email{golammostafaa@gmail.com}
	\thanks{}

	\date{\today}

	\begin{abstract}
		In this paper,
		we first prove that the totally real discs lying in certain Levi flat hypersurfaces are polynomially convex. As applications we prove that the totally real discs lying in the boundary of certain polynomial polyhedra are polynomially convex. We also provide an if and only if condition for polynomial convexity of totally real discs lying in the boundary of Hartog's triangle.
		We also provide sufficient conditions on general compact subsets lying on those hypersurfaces for polynomial convexity. 
	\end{abstract}

	\maketitle
	\section{Introduction}
	
	Let $K$ be a compact subset of the complex Euclidean space $\cplx^{n}.$ Let $\smoo(K)$ be the space of continuous complex valued functions on $K$ and $\poly(K)$ be the space of all uniform limits of polynomials on $K$. The polynomial convex hull of a compact set $K$ is denoted by $\widehat{K}$ and define by $\widehat{K}:=\{z\in\cplx^{n}:|p(z)|\le\sup_{K}|p|, \forall_{p}\in \cplx[z_1,\cdots,z_{n}]\}.$ $K$ is said to be polynomially convex if $\widehat{K}=K$. We say a compact $K\subset\cplx^n$ is rationally convex set if $K=\{z\in\cplx^{n}:|f(z)|\le\sup_{K}|f|, \forall f\in Rat(K)\}$, where $Rat(K)$ is the collection of all rational functions in $\cplx^n$ with poles outside $K$.
	
	One of the fundamental question in the theory of uniform algebras is to characterize compact subset $K$ of $\cplx^n$ for which 
	\begin{equation}\label{E:approx}
		\poly(K)=\smoo(K).
	\end{equation}
	Note that $\poly(K)$ and $\smoo(K)$ both are commutative Banach algebra and if $\poly(K)=\smoo(K)$ then their maximal ideal spaces are also same. From the theory of uniform algebra, we know that the maximal ideal space of $\smoo(K)$ can be identified with $K,$ and that of $\poly(K)$ can be identified with $\hat{K}.$ From the above discussion we can see that polynomial convexity arises naturally in the study of uniform approximation. Lavrentiev \cite{Lv} showed that for $K\subset \cplx,$ $\poly(K)=\smoo(K)$ if and only if $\hull{K}=K$ and $int(K)=\emptyset.$ No such characterization is known in the higher dimension, and it is generally challenging to decide which compacts of $\cplx^n$ satisfy (\ref{E:approx}). So it is natural to consider the problem for particular classes of compact sets. In this paper, we consider compact subsets of totally real submanifolds of $\cplx^2$.
	
	\par	Recall that a $C^{1}$-smooth submanifold $M$ of $\cplx^{n}$ is said to be \textit{totally real} at $p\in M$ if $T_{p}M\cap iT_{p}M=\{0\},$ where the tangent space $T_{p}M$ is viewed as a real linear subspace of $\cplx^{n}.$  The manifold $M$ is said to be totally real if it is totally real at all points of $M$. Totally real submanifolds play an important role because of the following reasons.
	\begin{enumerate}[(i)]
		\item Such manifolds are locally polynomially convex (see \cite{NirWell69}).
		\item  Let $K$ be a compact subset of a totally real manifold $M.$ Then any continuous function on $K$ can be uniformly approximated by a holomorphic (in a neighborhood of $K$) function on $K.$ If $K$ is polynomially convex, then the holomorphic function can be replaced by the polynomial (see \cite{HoW68}). Hence, \eqref{E:approx} holds in this case.
	\end{enumerate}
	For polynomially convex set $K$, there are several papers, for instance see \cite{AIW1, AIW2, OPW, St1, W1},
	that describe situations when \eqref{E:approx} holds. 
	\par	A closed \textit{totally real disc}\index{Totally real disc} is a compact subset of a $\smoo^1$ totally real submanifold, diffeomorphic to the closed planar disc. A totally real disc in $\cplx^2$ may not be polynomially convex as the following example shows.
	\begin{example}[Wermer \cite{HoW68}]\label{Exm:Wermer_Example}
		Let $M:=\{(z, f(z))\in \CC : z\in \cplx\},$ where
		$$f(z)=-(1+i)\bar{z}+iz\bar{z}^2+z^2\bar{z}^3.$$
		Then $M$ is totally-real. Let
		$K:=\{(z, f(z))\in \CC : z\in \overline{\dsc}\}\subset M.$ Then $K$ is not polynomially
		convex. 
	\end{example}
	
	Studies of polynomial convexity of totally-real discs which are of graph form is done by Wermer \cite{W1}, O'Farrell and Preskenis \cite{OP1,OP2} and Duval \cite{JDuval88}. Duval's theorem is particularly interesting. We mention it here:
	\begin{result}[Duval]
		Let $M=\{(z,f(z))\in\cplx^2: z\in\ba{\dsc}\}$, where $f$ is $\smoo^1$-smooth in a neighbourhood of $\dsc$. Assume $\left|\dfrac{\bdy f}{\bdy\zbar}(a)\right|>\left|\dfrac{\bdy f}{\bdy z}(a)\right|$ for all $a\in\ba{\dsc}$. Then $M$ is polynomially convex.
	\end{result}
	\noindent The study of totally real discs, which are not, in general, of graph form,  appeared in the context of removable singularities of CR-functions. A deep result in this direction is due to J\"{o}ricke \cite{Joric88}. 
	\begin{result}[J\"{o}ricke]
		Any $\smoo^2$-smooth totally real disc in the boundary of the unit ball in $\cplx^2$ is removable.
	\end{result}
	\noindent Combining a result due to Lupacciolu and Stout (see \cite{Sto93}) it is evident that any $\smoo^2$-smooth totally real disc in $\bdy \mathbb{B}^2$ is polynomial convex. 
	Stout mentioned in \cite{Sto93} that a similar argument will work for the $\smoo^2$-smooth totally real discs that lie in the boundary of a strictly pseudoconvex domain $\Omega$ such that $\ba{\Omega}$ is polynomially convex. 
	In \cite{Alex91} Alexander proved the following result: 
	\begin{result}[Alexander]\label{R:Alexander}
		Any $\smoo^2$-smooth totally real disc contained in $\{(z_1,z_2)\in \cplx^2:|z_1|=1\}$ is polynomially convex
	\end{result}
	\noindent After the works \cite{Joric88} and \cite{Alex91} were done, J\"{o}ricke, in the problem book \cite{HavNik06}, asked the following question:
	\begin{question}
		Which closed, totally real discs in $\cplx^2$ are polynomially convex?
	\end{question}
	\noindent The hypersurface that Alexander considered in \cite{Alex91} is Levi flat. This motivates us to look at some more examples of Levi-flat hypersurfaces and totally real discs lying in them.
	Recall that A hypersurface is said to be Levi flat if its Levi form vanishes identically.
	We also give partial answers to J\"{o}ricke's question in some cases.
	Our first result demonstrates a class of totally real discs in $\cplx^2$, lying in certain nonsingular Levi-flat hypersurfaces, that are polynomially convex. These generalizes Alexander's result. More precisely, we present:
	
	\begin{theorem}\label{T:Polynomial}
		Let $\Omega$ be a Runge domain in $\cplx^2$  and $h$ be a holomorphic function on $\Omega.$ 
		Let $M:=\{ z\in \Omega:|h(z)|=1\}$ with $dh(z)\not= 0$ for all $z\in M.$ Then every totally real 
		disc in $M$ is polynomially convex. 
	\end{theorem}
	
	\noindent We now mention a corollary which is somewhat interesting.
	\begin{corollary}\label{T:Holomorphic}
		Let $\Omega$ be a Runge domain in $\cplx^2$  and $g$ be a holomorphic function on $\Omega.$ 
		Let $M:=\{ z\in \Omega:\rl g(z)=0\}$ with $dg(z)\not= 0$ for all $z\in M.$ Then every totally real $\smoo^2$-smooth
		smooth disc in $M$ is polynomially convex. 
	\end{corollary}	
	Here we give another class of domains in $\cplx^2$, which lies outside the class of strictly pseudoconvex domains for which the totally real discs are polynomially convex.
	If $p_{1},\cdots,p_{l}$ are holomorphic polynomials in $\cplx^n,$ then 
	\begin{align*}
		\mathfrak{D}_{l}:=\{z\in \Omega:|p_1(z)|< 1,\cdots,|p_{l}(z)|< 1\}, l\in \mathbb{N},
	\end{align*}
	is known as a polynomial polyhedron in $\cplx^2$. Let us define
	\begin{itemize}
		\item $\Sigma_{j}:=\{z\in \cplx^n:|p_{j}(z)|=1\}$;
		\item $\Sigma_{J_k}:=\Sigma_{j_1}\cap\cdots\cap\Sigma_{j_k},$ where $J_k=\{j_1,\cdots,j_k\}$.
	\end{itemize}
	\noindent Clearly, the topological boundary $\partial\mathfrak{D}_{l}$ of $\mathfrak{D}_{l}$ is contained in $\cup_{j}\Sigma_{j}.$ The polynomial polyhedron  $\mathfrak{D}_{l}$ is said to be \textit{complex non-degenerate} if for any increasing collection $j_1<j_2<\cdots <j_k,k\le n,$ 
	\begin{align*}
		dp_{j_1}\wedge\cdots\wedge dp_{j_k}(z)\ne 0 ~\forall z\in \Sigma_{J_k}.
	\end{align*}
	
	\begin{corollary}\label{coro-polyhedron}
		Any totally real disc that lies in the non-singular part of the boundary of a complex non-degenerate polynomial polyhedron in $\cplx^2$ is polynomially convex.
	\end{corollary}
	
	\noindent The following is also a simple corollary of \Cref{T:Polynomial}. It deals with compact subsets that lies in the zero set of pluriharmonic functions. This needs simple connectedness of the domain as an assumption.
	
	\begin{corollary}\label{T:Pluriharmonic}
		Let $\Omega$ be a simply connected Runge domain in $\cplx^2$  and $h$ be a real valued pluriharmonic function on $\Omega.$ 
		Let $M:=\{ z\in \Omega: h(z)=0\}$ with $dh(z)\not= 0$ for all $z\in M.$ Then every totally real $\smoo^2$-smooth disc in $M$ is polynomially convex. 
	\end{corollary}	
	Another situation which quite different from the above is as follows:
	\begin{theorem}\label{T:Smooth_Cross_C}
		If $h:\cplx\to \mathbb{R}$ be a smooth map and $E_{h}:=\{x+iy\in \cplx:h(x,y)=0\}.$ If $dh\ne 0$ on $E_{h},$ then every $\smoo^2$-smooth totally real disc in $M:=E_{h}\times \cplx$ is polynomially convex.	
	\end{theorem}
	\noindent In \Cref{T:Holomorphic} if we take $\Om=\cplx^2$ and $g(z,w)=iz$ then we get:
	\begin{corollary}\label{Rmk:RXC}
		Let $M=\mathbb{R}\times\cplx$. Then every totally real $\smoo^2$-smooth disc in $M$ is polynomially convex.
	\end{corollary}	
	
	\noindent \Cref{Rmk:RXC} also follows from \Cref{T:Pluriharmonic} and \Cref{T:Smooth_Cross_C} independently.
	
	\smallskip
	
	In the second part of the paper, we will consider general compact subsets that lie inside certain Levi-flat hypersurfaces. The compacts now may not be totally real.
	We need some definition to state our results in this part. The following definitions can be found in Stolzenberg \cite{GStl63}.
	
	\begin{definition}
		Let $K$ be a compact subset of $\cplx^n.$ If $f:K\to \cplx\setminus \{0\}$ is of the form $f=\exp(\psi)$ for some map $\psi:K\to \cplx$ we say that $\ln (f)$ is defined and $\psi$ is a branch of $\ln(f).$
	\end{definition}	
	
	\begin{definition}
		Let $K$ be a compact subset of $\cplx^n.$ $K$ is said to be \emph{simply-coconnected}\index{Simply-coconnected} if, for every map $f:K\to \cplx\setminus \{0\},$ $\ln (f)$ is defined.
	\end{definition}	
	
	\begin{remark}[Stolzenberg]\label{Rmk:Simply_Coconntd_Cech}
		$K$ is simply-coconnected if and only if $\check{H}^{1}(K;\mathbb{Z})=0.$
	\end{remark}	
	\noindent	Hence, any contractible set is simply coconnected.
	\begin{definition}
		Let $V$ be an analytic variety in $\cplx^n.$ $V$ is said to be a \emph{Runge variety} if for every $K\subset V,$ $\widehat{K}\subset V.$ 
	\end{definition}

	\begin{definition}
		$K$ is said to be \emph{polynomially convex in dimension one}\index{Polynomially convex in dimension one} if, for every one-dimensional Runge variety $V$ such that $K\cap V$ is compact, $K\cap V$ is polynomially convex.
	\end{definition}
	\noindent	We are now in a position to present our first theorem in this part. 
	
	\begin{theorem}\label{T:Poly_Cnvx_OneDimntion}
		Let $P$ be a holomorphic polynomial in $\cplx^2.$ Let $M:=\{(z_1,z_2)\in \cplx^2:|P(z_1,z_2)|=1\}$ with $dP(z)\not= 0$ for all $z\in M,$ and $K\subset M$ be compact, simply-coconnected. If $P^{-1}\{c\}\cap K$ is polynomially convex for all $c\in \partial\mathbb{D},$ then $K$ is polynomially convex.
	\end{theorem}
	We also present a result analogous to \Cref{T:Smooth_Cross_C} for general compact sets.
	\begin{theorem}\label{T:smooth_cross-cgeneral}
		Let $h:\cplx\to \mathbb{R}$ be a smooth map and $E_{h}:=\{x+iy\in \cplx:h(x,y)=0\}$ with $dh\ne 0$ on $E_{h}.$ Let $K$ be a connected simply-coconnected compact subset of $M:=E_{h}\times \cplx$ such that each fiber $K_c=\{w\in\cplx: (c,w)\in K\}$ is polynomially convex. Then $K$ is polynomially convex.	
	\end{theorem}

	
	We now turn our attention to singular Levi-flat hypersurfaces. A singular hypersurface is said to be Levi flat if its regular part is Levi flat.
	We now provide an example of a non-polynomially convex compact subset which totally real except the singular point of the hypersurface. It shows that the singularities in the Levi-flat hypersurface have a role in determining polynomial convexity.
	
	\begin{example}
		Let us consider $M:=\{(z,w)\in \cplx^2:|z|=|w|\},$ $	K:=\{(z,w)\in \cplx^2:|z|=\rl w,0\le \rl w\le 1, \imag w=0\}\subset M.$ Then $K$ is not polynomially convex because $\{(z,w)\in \cplx^2:|z|\le 1, \rl w=1,\imag w=0\}\subset \widehat{K}.$ Note that $K$ is contractible. 
	\end{example} 
	
	\medskip
	Let $\Omega$ be a Runge domain in $\cplx^2$ and $h:\Omega\to \cplx$ be a holomorphic function and $M:=\{z\in \Omega:|h(z)|=1\}.$ Let $M^{*}:=\{z\in G:|h(z)|=1\}$ be the regular part of $M,$ where $G:=\Omega\setminus \{z\in \Omega:|h(z)|=1,dh(z)=0\}.$ By $M_{sing},$ we denote the singular part of $M.$
	
	\begin{theorem}\label{T:Singular_Levi_flat}
		Let $K$ be a totally real disc in $M^{*}.$ Then $K$ is polynomially convex if and only if $\widehat{K}\cap M_{sing}=\emptyset.$
	\end{theorem}
	\noindent As an application of \Cref{T:Singular_Levi_flat} we provide a necessary and sufficient condition for totally real discs in a particular type of Levi-flat hypersurface of the form $\{(z_1,z_2)\in \cplx^2: \rl (z_{1}^m+z^{n}_{2})=0\}$, $m, n\geq 2$, to be polynomially convex. For $n=2, m=2$ it is one of the normal form of the Levi-flat quadrics in the list given in \cite{DanielGong1999}.
	
	\begin{corollary}\label{Cor:SinLeviFlat_PolyCnvx}
		Every totally real disc $K$ in the singular Levi-flat hypersurface $M:=\{(z_1,z_2)\in \cplx^2: \rl (z_{1}^m+z^{n}_{2})=0\}$ is polynomially convex if and only if $(0,0)\notin \widehat{K}.$  
	\end{corollary}	
	Next we come back to the Levi-flat hypersurface $\{(z,w)\in \cplx^{2}:|z|=|w|\}$ and provide an if and only if condition for a totally real disc lying there to be polynomially convex. This hypersurface also provides another normal form of Levi-flat quadrics in $\cplx^2$ in the list of \cite{DanielGong1999}. 
	\begin{theorem}\label{T:Singular_Hypersurface}
		Every totally real disc $K$ in $M:=\{(z,w)\in \cplx^{2}:|z|=|w|\}\setminus\{(0,0)\}$ is polynomially convex if and only if $\widehat{K}\cap\{zw=0\}=\emptyset.$	
	\end{theorem}
	\noindent We now state a corollary in the setting of the boundary of Hartogs triangle.
	\begin{corollary}\label{Cor: Hartogsbdy}
		Every totally real disc $K$ in the nonsingular part of the boundary of the Hartogs triangle $\{(z_1,z_2)\in \cplx^2: |z_1|<|z_2|<1\}$ is polynomially convex if and only if $(0,0)\notin \widehat{K}.$
	\end{corollary}
	\noindent A proof of \Cref{Cor: Hartogsbdy} follows from \Cref{R:Alexander} and \Cref{T:Singular_Hypersurface}.
	
	\section{Technical Results}\label{S:technical}
	\noindent In this section we first collect some results from the literature those will be used in our proofs. The following result can be found in \cite[Corollary 27]{GStl63}.
	\begin{result}[Stolzenberg]\label{R:R-conx_PolyConvx}
		Every rationally convex simply-coconnected  set is polynomially convex in dimension one.
	\end{result}	
	
	\noindent We also make use of the following result \cite[Page 269, assertion (12)]{GStl63}.
	\begin{result}[Stolzenberg]\label{R:R_Convx_Runge_Varty}
		Let $K$ be a compact subset of the purely one-dimensional analytic variety $V$ in $\cplx^{n}$ such that $\widehat{K}\subset V.$ Then  $K$ is rationally convex. 
	\end{result}	
	
	\noindent The following result is also from Stolzenberg \cite{GStl63}.	
	
	\begin{result}[Stolzenberg]\label{R:Hull_Img X}
		Let $X$ be a compact subset of $\cplx^{n}.$ If $X$ is simply-coconnected and there is a function, $f$, holomorphic in a neighborhood of $\widehat{X}$ such that $f(X)\cap f(\widehat{X}\setminus X)=\emptyset,$ then $X$ is polynomially convex.
	\end{result}

	\noindent The following is due to Stolzenberg (see \cite[Lemma 2.3]{Gorai2014}).
	\begin{result}[Stolzenberg]\label{L:Gorai}
		Let $E$ be a compact subset of $\cplx^{n}.$ Assume that $\poly(E)$ contains a function $\phi$ such that $\phi(E)$ has empty interior and $\cplx\setminus\phi(E)$ is connected. Then $E$ is polynomially convex if and only if $\phi^{-1}(c)\cap E$ is polynomially convex for each $c\in \phi(E).$ 	
	\end{result}

	\noindent  The following result from Stout's book will also be useful in our proofs.
	\begin{result}(\cite[Lemma 1.6.18]{Sto07})\label{R:BdryPt_PeakPt}
		Let $X$ be a compact, polynomially convex subset of $\cplx.$ Then every point of $\partial X$(boundary of $X$) is a peak point for
		the algebra $\poly(X).$
	\end{result}
	\noindent The following result is due to Samuelsson and Wold \cite[Proposition 4.7]{SaW12}.	We will use this result crucially in our proofs.
	\begin{result}[Samuelsson-Wold]\label{R:Hull_fiber}
		Let $X$ be a compact subset of $\cplx^{n}$ and $F:\cplx^{n}\to\cplx^{m}$ be the uniform limit on $X$ of entire functions. Let $K=F(X).$ If $\alpha\in K$ is a peak point for the algebra $\poly(K),$ then 
		\begin{align*}
			F^{-1}\{\alpha\}\cap\widehat{X}=\widehat{F^{-1}\{\alpha\}\cap X}.
		\end{align*}
	\end{result}

	Next we mention couple of standard result from the theory of ordinary differential equations.
	Consider the system of differential equation
	
	\begin{align}\label{E:Diff eqn}
		\begin{cases}
			\frac{dx}{dt}&=F(x,y)\\
			\frac{dy}{dt}&=G(x,y).
		\end{cases}
	\end{align}
	
	\begin{result}[Poincar\'{e}]\label{R:Poincare}
		Every closed path of the system (\ref{E:Diff eqn}) necessarily surrounds at least one critical point.
	\end{result}
	
	\begin{result}[Poincar\'{e}-Bendixson]\label{R:Poincare-Bendixson}\index{Poincar\'{e}-Bendixson Theorem}
		Let $\Omega$ be a bounded region of the phase plane together its boundary, and assume that $\Omega$ does not contain any critical points of the system (\ref{E:Diff eqn}). If $V:=(x(t),y(t))$ is a path of (\ref{E:Diff eqn}) that lies in $\Omega$ for all $t\ge t_{0},$ then $V$ is either itself a closed path or it spirals toward a closed path as $t\rightarrow\infty.$ Thus in either case the system (\ref{E:Diff eqn}) has closed path in $\Omega.$
	\end{result}

	Next we state and prove a few lemmas that will be used in our proofs.
	The first one is easy but useful in our context.
	\begin{lemma}\label{L:Alg_Variety_Runge}
		Let $\Omega$ be a Runge domain in $\cplx^n$ and $f$ be a holomorphic function on $\Omega.$ Then every level set of $f$ is a Runge variety.	
	\end{lemma}
	\begin{proof}
		Let $Z_{c}:=\{z\in \Omega:f(z)=c\},$ where $c\in \partial\mathbb{D}.$ 
		Let $K$ be any compact subset of $Z_{c}.$ We claim that the polynomial convex hull $\widehat{K}$ of $K$ is also subset of $Z_{c}.$ Since, $\Omega$ is Runge, $\widehat{K}\Subset \Omega.$
		Suppose that $\alpha\in \Omega\setminus Z_{c}.$  Then,  $f(\alpha)\ne c.$ Define $g(z):=f(z)- c.$ Therefore,
		\begin{align*}
			|g(\alpha)|>0=\sup_{K}|g|.
		\end{align*}
		\noindent Since $\Omega$ is Runge, $g$ can be approximated by holomorphic polynomials on $K.$
		Hence $\alpha\notin \widehat{K}$ i.e. $\widehat{K}\subset Z.$	
	\end{proof}

	\par Let $M$ be a three-dimensional smooth submanifold of $\cplx^2$ and $N$ be a two-dimensional totally real submanifold of $M$ with (possibly empty) boundary. A field of lines on $N$ is defined as follows: for every $p\in  N,$
	\begin{align*}
		L_{p}=T_{p}N\cap T^{\cplx}_{p}M.
	\end{align*}
	Here, $T^{\cplx}_{p}M$ is the unique complex line through $p$ that is tangent to $M$ (where $T^{\cplx}_{p}M=T_{p}M\cap iT_{p}M$). The curves that are always tangent to the lines $L_{p}$ are the leaves of the foliation of $N$ that is defined by this line field. In other words, a curve $\gamma:(a,b)\to M$ is included in a leaf of the foliation if and only if the derivative $\gamma'(t)$ falls on the line $L_{\gamma(t)}$ for every $t\in (a,b).$ This gives a foliation of class $\smoo^{1}$ of $N,$ also known as the \textit{characteristic foliation} \index{Characteristic foliation} of $N$ 
	(see Stout \cite[Page 258]{Sto07} for details).
	

	\medskip
	
	\begin{lemma}\label{L:Charc_Foliation}
		Let $\Omega\subset \cplx^2$ be a domain and $\rho:\Omega\to \mathbb{R}$ is $\smoo^2$-smooth and $d\rho\ne 0$ on $M:=\{z\in \Omega:\rho(z)=0\}.$
		Let $K$ be a totally real $\smoo^2$-smooth disc in the 3-dimensional manifold $M.$ Then there exists a characteristic foliation on $K.$
	\end{lemma}
	\begin{proof}
		Since $K$ is a totally real disc in $M,$ there exists a neighborhood $U$ of $\mathbb{\overline{D}}$ and a totally real smooth submanifold $N$ of $\cplx^2$ and a $C^{2}$-smooth diffeomorphism $\phi:U\to N$ such that $\phi(\mathbb{\overline{D}})=K.$ By shrinking $U$ (if necessary), we can assume $U$ is also contractible and we take $\triangle:=\phi(U).$ Therefore, $\triangle$ is also a $\smoo^2$-smooth totally real disc in $M$ which contains $K.$ We take 
		\begin{align*}
			\phi(t,s)=(\phi_1(t,s),\phi_2(t,s),\phi_3(t,s),\phi_4(t,s)).
		\end{align*}
		\noindent Therefore,
		$$d\phi|_{(t,s)}=
		\begin{pmatrix} 
			\frac{\partial \phi_{1}}{\partial{t}}(t,s) & \frac{\partial \phi_{1}}{\partial{s}}(t,s) \\[1.5ex]
			\frac{\partial \phi_{2}}{\partial{t}}(t,s)&  \frac{\partial \phi_{2}}{\partial{s}}(t,s)\\[1.5ex]
			\frac{\partial \phi_{3}}{\partial{t}}(t,s)&  \frac{\partial \phi_{3}}{\partial{s}}(t,s)\\[1.5ex]
			\frac{\partial \phi_{4}}{\partial{t}}(t,s) &  \frac{\partial \phi_{4}}{\partial{s}}(t,s)\\\\
		\end{pmatrix}_{4\times 2}=\begin{pmatrix}
			&\\A(t,s) & B(t,s)\\\\
		\end{pmatrix},$$ where
		$$A(t,s)=\begin{pmatrix} 
			\frac{\partial \phi_{1}}{\partial{t}}(t,s) \\[1.5ex]
			\frac{\partial \phi_{2}}{\partial{t}}(t,s)\\[1.5ex]
			\frac{\partial \phi_{3}}{\partial{t}}(t,s)\\[1.5ex]
			\frac{\partial \phi_{4}}{\partial{t}}(t,s)\\\\
		\end{pmatrix} \text{and }B(t,s)=	\begin{pmatrix} 
			\frac{\partial \phi_{1}}{\partial{s}}(t,s) \\[1.5ex]
			\frac{\partial \phi_{2}}{\partial{s}}(t,s)\\[1.5ex]
			\frac{\partial \phi_{3}}{\partial{s}}(t,s)\\[1.5ex]
			\frac{\partial \phi_{4}}{\partial{s}}(t,s)\\\\
		\end{pmatrix}.$$
		
		Let $\lambda=a(t,s)\frac{\partial}{\partial t}+b(t,s)\frac{\partial}{\partial s}$ be a non-vanishing vector field on $U.$
		
		Now we compute, for smooth $f$ on $K,$
		\begin{align}\label{E:Pushfrwrd_Vfld}
			d\phi(\lambda)(f)&=\lambda(f\circ \phi)\
			=\left(a(t,s)\frac{\partial}{\partial t}+b(t,s)\frac{\partial}{\partial s}\right)(f\circ \phi).
		\end{align}
		\noindent Now, by chain rule, we get that
		\begin{align*}
			\frac{\partial (f\circ \phi)}{\partial t}=\left(\frac{\partial \phi_{1}}{\partial t}\frac{\partial }{\partial x}\bigg|_{\phi(t,s)}+\frac{\partial \phi_{2}}{\partial t}\frac{\partial }{\partial y}\bigg|_{\phi(t,s)}+\frac{\partial \phi_{3}}{\partial t}\frac{\partial }{\partial u}\bigg|_{\phi(t,s)}+\frac{\partial \phi_{4}}{\partial t}\frac{\partial }{\partial v}\bigg|_{\phi(t,s)}\right)(f),
		\end{align*}
		
		\noindent and
		\begin{align*}
			\frac{\partial (f\circ \phi)}{\partial s}=\left(\frac{\partial \phi_{1}}{\partial s}\frac{\partial }{\partial x}\bigg|_{\phi(t,s)}+\frac{\partial \phi_{2}}{\partial s}\frac{\partial }{\partial y}\bigg|_{\phi(t,s)}+\frac{\partial \phi_{3}}{\partial s}\frac{\partial }{\partial u}\bigg|_{\phi(t,s)}+\frac{\partial \phi_{4}}{\partial s}\frac{\partial }{\partial v}\bigg|_{\phi(t,s)}\right)(f).
		\end{align*}
		
		\noindent Since $\lambda$ is a non-vanishing vector-field on $U$ and $\phi$ is a diffeomorphism,	therefore, $\nu =a(t,s)A(t,s)+b(t,s)B(t,s)$ is also a non-vanishing vector field on $\triangle.$
		
		\medskip

		\noindent Now we will see that the above vector field $\nu$ gives a $\smoo^{1}$-smooth characteristic foliation on $K.$ 	
		To see this we do some computations here:
		Let $\nu_{\alpha}\in T_{\alpha}^{\cplx} M.$ Then
		\begin{align*}
			& \langle J(a(t,s)A(t,s)+b(t,s)B(t,s),\nabla \rho)\rangle=0 ~, \\
			&\implies a(t,s)\langle JA(t,s),\nabla \rho\rangle+b(t,s)\langle JB(t,s),\nabla \rho\rangle=0,
		\end{align*}
		where $J$  is the standard complex structure on $\cplx^2,$ and $\langle,\rangle$ is the standard real inner product. A natural choice for the functions $a$ and $b$ is the following:
		
		\begin{align}\label{E:Choice a,b}
			\begin{cases}
				a(t,s)=\langle JB(t,s),\nabla \rho\rangle~ \text{ and}\\  b(t,s)=-\langle JA(t,s),\nabla \rho\rangle.
			\end{cases}	
		\end{align}

		\noindent From (\ref{E:Choice a,b}), we get that
		\begin{align*}
			a(t,s)&=-\frac{\partial\phi_{2}}{\partial s}\frac{\partial\rho}{\partial x}+\frac{\partial\phi_{1}}{\partial s}\frac{\partial\rho}{\partial y}-\frac{\partial\phi_{4}}{\partial s}\frac{\partial\rho}{\partial u}+\frac{\partial\phi_{3}}{\partial s}\frac{\partial\rho}{\partial v}~~~\text{and }\\
			b(t,s)&=\frac{\partial\phi_{2}}{\partial t}\frac{\partial\rho}{\partial x}-\frac{\partial\phi_{1}}{\partial t}\frac{\partial\rho}{\partial y}+\frac{\partial\phi_{4}}{\partial t}\frac{\partial\rho}{\partial u}-\frac{\partial\phi_{3}}{\partial t}\frac{\partial\rho}{\partial v}.\\
		\end{align*}
		
		\noindent Since $\phi$ is $C^{2}$-smooth and $\rho$ is $C^{\infty}$ smooth, therefore from the above calculation we get that $a(t,s)$ and $b(t,s)$ are smooth functions.
		We assign each point $\alpha\in K$ to $\nu_{\alpha}\in \mathcal{L}_{\alpha}:=T_{\alpha}K\cap T_{\alpha}^{\cplx} M.$ Since $K$ is totally real disc in three dimensional manifold $M,$ dimension of $\dim_{\mathbb{R}}\mathcal{L}_{\alpha}=1~\forall \alpha\in K.$ Therefore, the above assignment is smooth one dimensional distribution on $\triangle.$ Since every one dimensional distribution is integrable and hence gives the characteristic foliation on $K.$
	\end{proof}

	\begin{lemma}\label{L:Func_Constant_Leaf}
		If $\rho(z):=|h(z)|^2-1,$ where $h$ is a holomorphic function in a neighborhood of $K$ in the \Cref{L:Charc_Foliation}, then $h$ is constant along each leaf of the characteristic foliation 
		of $K.$
	\end{lemma}	
	\begin{proof}
		Let the integral curve $\gamma(t):=\left(\gamma_1(t)+i\gamma_2(t),\gamma_3(t)+i\gamma_4(t)\right):(a,b)\to \triangle$ in $\triangle$ be a leaf of the characteristic foliation on $\triangle.$ Then $\gamma'(t)\in T_{\gamma(t)}K\cap T_{\gamma(t)}^{\cplx} M~~\forall t\in (a,b).$ Note that $\nabla\rho|_{\gamma(t)}=\left(\frac{\partial\rho}{\partial x},\frac{\partial\rho}{\partial y},\frac{\partial\rho}{\partial u},\frac{\partial\rho}{\partial v}\right)|_{\gamma(t)}$ is the normal to $M$ at $\gamma(t).$ Since $\gamma'(t)\in T_{\gamma(t)}^{\cplx} M,$ $i\gamma'(t)\in T_{\gamma(t)}M.$ Therefore,
		\begin{align*}
			\langle\gamma'(t),\nabla\rho|_{\gamma(t)}\rangle=0=\langle i\gamma'(t),\nabla\rho|_{\gamma(t)}\rangle.
		\end{align*}
		\noindent	Hence $\langle\gamma'(t),\nabla\rho|_{\gamma(t)}\rangle=0$ implies 
		\begin{align}\label{E:Gamma_tangent}
			&\gamma'_1(t)\frac{\partial \rho(\gamma(t))}{\partial x}+\gamma'_2(t)\frac{\partial \rho(\gamma(t))}{\partial y}+\gamma'_3(t)\frac{\partial \rho(\gamma(t))}{\partial u}+\gamma'_4(t)\frac{\partial \rho(\gamma(t))}{\partial v}=0,
		\end{align}
		\noindent and $\langle i\gamma'(t),\nabla\rho|_{\gamma(t)}\rangle=0$ implies
		\begin{align}\label{E:iGamma_tangent}	
			& -\gamma'_2(t)\frac{\partial \rho(\gamma(t))}{\partial x}+\gamma'_1(t)\frac{\partial \rho(\gamma(t))}{\partial y}-\gamma'_4(t)\frac{\partial \rho(\gamma(t))}{\partial u}+\gamma'_3(t)\frac{\partial \rho(\gamma(t))}{\partial v}=0.
		\end{align}		
		
		\noindent We get from (\ref{E:Gamma_tangent}) and (\ref{E:iGamma_tangent}) that
		\begin{multline*}
			\left(\gamma'_1(t)+i\gamma'_2(t)\right)\frac{\partial \rho(\gamma(t))}{\partial z}+\left(\gamma'_3(t)+i\gamma'_4(t)\right)\frac{\partial \rho(\gamma(t))}{\partial w}+\left(\gamma'_1(t)-i\gamma'_2(t)\right)\frac{\partial \rho(\gamma(t))}{\partial \bar{z}}\\+\left(\gamma'_3(t)-i\gamma'_4(t)\right)\frac{\partial \rho(\gamma(t))}{\partial \bar{w}}=0 ~
		\end{multline*}	and 
		\begin{multline*}
			i\left(\gamma'_1(t)+i\gamma'_2(t)\right)\frac{\partial \rho(\gamma(t))}{\partial z}+i\left(\gamma'_3(t)+i\gamma'_4(t)\right)\frac{\partial \rho(\gamma(t))}{\partial w}-i\left(\gamma'_1(t)-i\gamma'_2(t)\right)\frac{\partial \rho(\gamma(t))}{\partial \bar{z}}\\-i\left(\gamma'_3(t)-i\gamma'_4(t)\right)\frac{\partial \rho(\gamma(t))}{\partial \bar{w}}=0.
		\end{multline*}
		
		\noindent This implies 		
		\begin{align*}
			&\left(\gamma'_1(t)+i\gamma'_2(t)\right)\frac{\partial \rho(\gamma(t))}{\partial z}+\left(\gamma'_3(t)+i\gamma'_4(t)\right)\frac{\partial \rho(\gamma(t))}{\partial w}=0\\
			\implies&\left(\gamma'_1(t)+i\gamma'_2(t)\right)\frac{\partial h(\gamma(t))}{\partial z}+\left(\gamma'_3(t)+i\gamma'_4(t)\right)\frac{\partial h(\gamma(t))}{\partial w}=0\\
			\implies & d(h\circ \gamma)(t)=0\\
			\implies& h\circ\gamma=c, \text{ where c is some constant}.
		\end{align*}
		Therefore, we can say that each integral curve of $\nu$ in $K$ lies in $h^{-1}\{c\}\cap K$ for some constant $c.$
	\end{proof}

	\section{Nonsingular Levi-flat hypersurfaces}
	In this section we provide proofs of \Cref{T:Polynomial}, \Cref{T:Holomorphic}, \Cref{coro-polyhedron}, \Cref{T:Pluriharmonic}, \Cref{T:Poly_Cnvx_OneDimntion}, \Cref{Rmk:RXC} and \Cref{T:smooth_cross-cgeneral}.
	We first define $\rho(z):=h(z)\overline{h(z)}-1$ and $K$ be a smooth totally real disc in $M:=\{z\in \Omega: \rho(z)=0\}.$ Since $dh(z)\not= 0,$ $M$ is a real 3-dimensional submanifold of $\Omega \subset \cplx^{2}.$ Then the complex dimension of the complex tangent space $T^{\cplx}_{\alpha}M$ of $M$ is $1$ for each $\alpha\in M.$
	\begin{proof}[Proof of \Cref{T:Polynomial}]	
		We break our proof of \Cref{T:Polynomial} into the following steps.
		
		\medskip		
		\noindent {\bf Step I:} {\bfseries\boldmath Showing that each fiber $h^{-1}\{c\}\cap K$ is polynomially convex.}
		
		\medskip
		Since $dh|_{\alpha}\not =0$ on $h^{-1}\{c\},$ $\{z\in \Omega
		:h(z)=c\}$ is a complex manifold of pure dimension 1. In view of \Cref{L:Alg_Variety_Runge}, we can say that $h^{-1}\{c\}$ is a Runge variety for each $c\in \partial\mathbb{D},$ and we denote $K_{c}:=h^{-1}\{c\}\cap K.$ The following two cases hold.
		
		\smallskip
		
		\noindent  {\bf Case I: $ \check{H}^{1}(K_{c},\mathbb{Z})=0.$} This implies $K_{c}$ is simply-coconnected (see \Cref{Rmk:Simply_Coconntd_Cech}). Since $h^{-1}\{c\}$ is a Runge variety, by using \Cref{R:R_Convx_Runge_Varty}, we get that $K_{c}$ is rationally convex. Therefore, by using \Cref{R:R-conx_PolyConvx}, we can say that $K_{c}$ is polynomially convex in dimension one. Since $h^{-1}\{c\}$ is a one-dimensional Runge variety, therefore, $K_{c}\cap h^{-1}\{c\}=K_{c}$ is polynomially convex.

		\smallskip	
		
		\noindent  {\bf Case II: $ \check{H}^{1}(K_{c},\mathbb{Z})\ne 0.$} The leaves of the characteristic foliation of $\triangle$ are curves $\eta$ such that for all $t$, $\eta'(t)\in T_{\eta(t)}K\cap T_{\eta(t)}^{\cplx} M.$ On the disc $\triangle,$ there are nowhere vanishing vector fields $V$ (see \Cref{L:Charc_Foliation}) such that the vector $V_{x}$ is tangent to the leaf through $x$ of the characteristic foliation at all points $x\in \triangle.$ Thus, the solution curves of the differential equation $x'=V_{x}$ are contained in leaves of
		the foliation. According to the Poincar\'{e}-Bendixson theorem (\Cref{R:Poincare-Bendixson}), all of the integral curves for this foliation are arcs, homeomorphic to the interval $[0,1],$ with endpoints on $\partial\triangle.$ In fact, let $\gamma$ be a integral curve for this foliation. If $\gamma$ lies in $\triangle,$ then there exists a closed integral curve $\widetilde{\gamma}$ for a non-vanishing vector field $\widetilde{V}$ on $U,$ (where $\phi(U)=\triangle$). Therefore, by \Cref{R:Poincare-Bendixson}, either $\widetilde{\gamma}$ is closed or it spirals toward a closed path as $t\to \infty.$ Therefore, by \Cref{R:Poincare}, there exists a critical point in $U.$ This is a contradiction because $\widetilde{V}$ is nowhere vanishing vector field on $U.$ Note that $\triangle$ contain all fibers $h^{-1}\{c\}\cap K.$ Since $\check{H}^{1}(K_{c},\mathbb{Z})\not =0,$ $\triangle\setminus K_{c}$ is not connected. Take $\alpha\in K\setminus K_{c}$ and $\gamma$ is integral curve passing through $\alpha.$ We claim that $\gamma\cap K_{c}=\emptyset:$ if possible, assume that $\gamma\cap K_{c}\ne \emptyset$ and $\xi\in \gamma\cap K_{c}.$ Since each integral curve lies in $h^{-1}\{c_1\}$ for some constant $c_1\in\partial\mathbb{D},$ let $\gamma\subset h^{-1}\{c_1\}.$ Then we have $h(\alpha)=c_{1}.$ On the other hand $h(\xi)=c$ and $\xi\in \gamma.$ But we know that $h$ is constant along each integral curve (by \Cref{L:Func_Constant_Leaf}). This is a contradiction to the assumption that $\gamma\cap K_{c}\not=\emptyset.$ Therefore, $\gamma\cap K_{c}=\emptyset.$ This proves the claim. Next we claim that $\gamma\cap K_{c}=\emptyset$ is not possible. By Poincar\'{e}-Bendixson theorem (\Cref{R:Poincare-Bendixson}) each integral curve $\gamma$ through $\alpha$ joins $\alpha$ to the boundary of $\triangle.$ But $\triangle\setminus K_{c}$ is not connected, hence $\gamma\cap K_{c}=\emptyset$ is not possible. Therefore, $ \check{H}^{1}(K_{c},\mathbb{Z})\ne 0$ is not possible, i.e. {\bf Case II} can not happen. 
		
		\medskip
		
		\noindent {\bf Step II:} {\bfseries\boldmath Completing the proof.}
		
		\medskip	
		We claim that $h(K)\cap h(\widehat{K}\setminus K)=\emptyset.$ If possible, assume that $h(K)\cap h(\widehat{K}\setminus K)\not =\emptyset.$ Let $\alpha\in h(K)\cap h(\widehat{K}\setminus K).$ Then there exist $\eta\in K$ and $\xi\in \widehat{K}\setminus K$ such that $h(\eta)=\alpha$ and $h(\xi)=\alpha.$	
		We claim that  $\xi\in \widehat{h^{-1}\{\alpha\}\cap K}$ that is the fiber $\widehat{h^{-1}\{\alpha\}\cap K}$ is not polynomially convex. Let $Y=h(K).$ Since $h(K)\subset \partial\mathbb{D},$ each point of $Y$ is a peak point for the uniform algebra $\poly(Y),$ then by \Cref{R:Hull_fiber}, we obtain that
		\begin{align*}
			h^{-1}\{\alpha\}\cap\widehat{K}=\widehat{h^{-1}\{\alpha\}\cap K}.	
		\end{align*}
		\noindent Therefore,		
		\begin{align*}
			\xi\in h^{-1}\{\alpha\}\cap\widehat{K}=\widehat{h^{-1}\{\alpha\}\cap K}.	
		\end{align*}	
		Hence, 		
		$h^{-1}\{\alpha\}\cap K$ is not polynomially convex. This is a contradiction to {\bf Step I}. Hence, $h(K)\cap h(\widehat{K}\setminus K)=\emptyset.$ Therefore, by \Cref{R:Hull_Img X}, $K$ is polynomially convex.	
	\end{proof}

	\begin{proof}[Proof of \Cref{T:Holomorphic}]
		Proof of \Cref{T:Holomorphic} follows from \Cref{T:Polynomial}. To see this, we define		$h(z):=e^{g(z)}$ on $\Omega.$ Then $\{z\in \Omega:|e^{g(z)}|=1\}=\{z\in \Omega:\rl g(z)=0\}=
		M.$	Since $\Omega$ is Runge and $e^{g(z)}$ is holomorphic on $\Omega,$ therefore every totally real smooth disc contained in $\{z\in \Omega:|e^{g(z)}|=1\}$  is polynomially convex. This proves the theorem.
	\end{proof}	
	
	\begin{proof}[Proof of \Cref{T:Pluriharmonic}]
		Proof of \Cref{T:Pluriharmonic} easily follows from \Cref{T:Holomorphic}. Since $\Omega$ is simply-connected, there exists a holomorphic function $P$ on $\Omega$ such that $h(z)=\rl P(z).$ Using \Cref{T:Holomorphic}, we can say that every smooth totally real disc in $\{z\in \Omega:\rl P(z)=0\}$ is polynomially convex. Hence,  every smooth totally real disc in $\{z\in \Omega: h=\rl P(z)=0\}$ is polynomially convex.
	\end{proof}

	\begin{proof}[Proof of \Cref{T:Poly_Cnvx_OneDimntion}]
		Since $K$ is polynomially convex in dimension one, $h^{-1}\{c\}\cap K$ is polynomially convex for all $c\in \partial\mathbb{D}.$ Hence $h(K)\cap h(\widehat{K}\setminus K)=\emptyset$ (see {\bf Step II} of the proof of \Cref{T:Polynomial}). Again $K$ is simply-coconnected, therefore, by \Cref{R:Hull_Img X}, we can say that $K$ is polynomially convex.
	\end{proof}
	
	\begin{proof}[Proof of Theorem~\ref{T:Smooth_Cross_C}]
		First, we claim that integral curves for the characteristic foliation for $M$ are the curves contained in $P^{-1}\{c\}$ for some $c,$ where $P:\cplx^2\to \cplx, (z,w)\mapsto z$ is the projection map. To see this, let $(x,y,u,v)$ be the coordinates of $\mathbb{R}^{4}=\cplx^2.$ Let us define $\rho:\cplx^2\to \mathbb{R}$ by $\rho(x,y,u,v)=h(x,y),$ then $M=\rho^{-1}\{0\}.$ The normal vector to $M$ at $(x,y,u,v)$ is given by
		\begin{align}\label{E:Smooth_Normal to M}
			\nabla\rho|_{(x,y,u,v)}=\left(\frac{\partial\rho}{\partial x},\frac{\partial\rho}{\partial y},\frac{\partial\rho}{\partial u},\frac{\partial\rho}{\partial v}\right)\bigg|_{(x,y,u,v)}=\left(\frac{\partial h}{\partial x},\frac{\partial h}{\partial y},0,0\right).
		\end{align}

		\medskip
		
		\noindent Let $\gamma(t)=(x(t), y(t), u(t), v(t))$ be a integral curve for the characteristic foliation for $M.$ Since $\gamma'(t)\in T_{\gamma(t)}^{\cplx} M,$ $i\gamma'(t)\in T_{\gamma(t)}M,$ therefore,
		\begin{align*}
			\langle\gamma'(t),\nabla\rho|_{\gamma(t)}\rangle=0=\langle i\gamma'(t),\nabla\rho|_{\gamma(t)}\rangle.
		\end{align*}
		\noindent	Hence $\langle\gamma'(t),\nabla\rho|_{\gamma(t)}\rangle=0$ implies 
		\begin{align}\label{E:Smooth_Ellipse_tangent}
			\frac{\partial h}{\partial x}x'(t)+\frac{\partial h}{\partial y}y'(t)=0,
		\end{align}
		\noindent and $\langle i\gamma'(t),\nabla\rho|_{\gamma(t)}\rangle=0$ implies
		\begin{align}\label{E:Smooth_Ellipse_itangent}	-\frac{\partial h}{\partial x}y'(t)+\frac{\partial h}{\partial y}x'(t)=0.
		\end{align}		
		\noindent From (\ref{E:Smooth_Ellipse_tangent}) and (\ref{E:Smooth_Ellipse_itangent}), we get that
		$$
		\begin{pmatrix} 
			\frac{\partial h}{\partial x} & \frac{\partial h}{\partial y} \\[1.5ex]
			\frac{\partial h}{\partial y} & -\frac{\partial h}{\partial x}\\
		\end{pmatrix}\begin{pmatrix}
			x'(t) \\
			y'(t)\\
		\end{pmatrix}=\begin{pmatrix}
			0 \\
			0\\
		\end{pmatrix}.$$
		\noindent Since $\left(\frac{\partial h}{\partial x}\right)^2+\left(\frac{\partial h}{\partial y}\right)^2\ne 0$ on $M,$ therefore $x'(t)=0=y'(t).$ This implies $x(t)=c_1,~~ y(t)=c_2$ for some $(c_1,c_2)\in E_{h}.$ Hence $\gamma\subset P^{-1}\{c\}$ which proves our claim.
		
		\medskip
		Let $K$ be a smooth totally real disc in $M.$ Following the proof of \Cref{T:Polynomial}, we can prove that each fiber $K_{c}=P^{-1}(c)\cap K=(\{c\}\times\cplx)\cap K$ is polynomially convex. To see this, let us define $K_{c}:=P^{-1}\{c\}\cap K$ and $V_{c}=\{(z,w)\in \cplx^2:z=c\}.$
		
		\medskip
		
		\noindent  {\bf Case I: $ \check{H}^{1}(K_{c},\mathbb{Z})=0.$} This implies $K_{c}$ is simply-coconnected (see \Cref{Rmk:Simply_Coconntd_Cech}). Since $V_{c}$ is a Runge variety, by using \Cref{R:R_Convx_Runge_Varty}, we get that $K_{c}$ is rationally convex. Therefore, by using \Cref{R:R-conx_PolyConvx}, we can say that $K_{c}$ is polynomially convex in dimension one. Since $V_{c}$ is a one-dimensional Runge variety, therefore, $K_{c}\cap V_{c}=K_{c}$ is polynomially convex. 
		
		\medskip
		
		\noindent {\bf Case II: $ \check{H}^{1}(K_{c},\mathbb{Z})\ne 0$}. By the similar argument as \Cref{T:Polynomial} ({\bf Step I}, {\bf Case II}), $ \check{H}^{1}(K_{c},\mathbb{Z})\ne 0$ can not happen.
		\medskip
		
		\noindent Therefore, 	$K_{c}=P^{-1}\{c\}\cap K$ is polynomially convex. $K$ is polynomially convex then follows from \Cref{R:Hull_Img X}.
	\end{proof}

	\begin{proof}[Proof of \Cref{Rmk:RXC}]
		Let $\Omega=\cplx^2$ and $f(z_1)=e^{-iz_1},$ Then $M:=\{(z_1,z_2)\in \cplx^2:|e^{-iz_1}|=1\}:=\mathbb{R}\times \cplx.$	Proof of \Cref{Rmk:RXC} now follows from \Cref{T:Polynomial}.
		
		\medskip
		Alternatively, if we take $h(z_1,z_2)=\imag z_1,$ then $\mathbb{R}\times \cplx=\{z\in \cplx^2:h(z)=0\}.$ Now proof follows from \Cref{T:Smooth_Cross_C}.
	\end{proof}
	
	\begin{proof}[Proof of \Cref{T:smooth_cross-cgeneral}]
		Let $P:\cplx^2\to \cplx, P(z,w)=z$ be the projection map. Since $K$ is polynomially convex in dimension one, $P^{-1}\{c\}\cap K$ is polynomially convex for all $c\in \cplx.$ We now claim that each point of $Y:=P(K)$ is a peak point for the algebra $\poly(Y).$ First, we show that $Y\subset \{\alpha\in \cplx:h(\alpha)=0\}:$ Let $\alpha\in Y,$ then there exists $(\xi,\eta)\in K\subset E_{h}\times \cplx$ such that $P(\xi,\eta)=\alpha.$ This implies $\xi=\alpha$ and since $(\xi,\eta)\in E_{h}\times \cplx,$ $h(\xi)=0$ i.e., $h(\alpha)=0.$ Therefore, $Y\subset \{\alpha\in \cplx:h(\alpha)=0\}.$
		
		\medskip
		\noindent Since $dh\ne 0$ on $E_{h},$ hence $E_{h}$ is a smooth submanifold of $\cplx$ of real dimension one. Note that $Y=P(K)$ is a connected subset of $E_{h}.$
		
		\medskip
		\noindent {\bf Case I}: $\cplx\setminus P(K)$ is connected, then $P(K)$ is polynomially convex. By \Cref{R:BdryPt_PeakPt}, we can say that each point of $\partial P(K)=P(K)$ is a peak point for the algebra $\poly(P(K)).$
		
		\medskip
		\noindent {\bf Case II}: $\cplx\setminus P(K)$ is not connected. Then $\widehat{P(K)}$ is the union of $P(K)$ with all bounded components of $\cplx\setminus P(K).$ Since $P(K)$ is connected, hence $\partial \widehat{P(K)}=P(K).$ By \Cref{R:BdryPt_PeakPt}, we can say that each point of $P(K)$ is a peak point for the algebra $\poly(P(K)).$

		By \Cref{R:Hull_fiber}, we get that \begin{align*}
			P^{-1}\{c\}\cap\widehat{K}=\widehat{P^{-1}\{c\}\cap K} 
		\end{align*} 
		for all $c\in \cplx.$ Therefore, $P(K)\cap P(\widehat{K}\setminus K)=\emptyset.$ Since $K$ is simply-coconnected, by \Cref{R:Hull_Img X}, $K$ is polynomially convex.	
	\end{proof}	
	

	
	\section{Singular Levi-flat hypersurfaces}
	In this section we provide the proofs of the theorems concerning singular Levi-flat hypersurfaces. First we give a proof of \Cref{T:Singular_Levi_flat}.
	\begin{proof}[Proof of Theorem~\Cref{T:Singular_Levi_flat}]
		If $K$ is polynomially convex, then obviously $\widehat{K}\cap M_{sing}=\emptyset.$
		
		\medskip
		We now prove the converse part. Since $dh|_{\alpha}\not =0,$ on $h^{-1}\{c\}\cap G,$ then $\{z\in G:h(z)=c\}$ is a complex manifold of pure dimension $1.$ By \Cref{L:Charc_Foliation}, there exist characteristic foliation $K,$ and by \Cref{L:Func_Constant_Leaf}, we get that the function $h(z,w)$ is constant along each leaf of the characteristic foliation of $\triangle.$
		
		\medskip
		\noindent {\bf Step I:} {\bfseries\boldmath Showing that each fiber $h^{-1}\{c\}\cap G\cap K$ is polynomially convex.}
		
		\medskip	
		We define $K_{c}:=P^{-1}\{c\}\cap G\cap K$ and $V_{c}=\{z\in \cplx^2:h(z)=c\}.$ Since $K\cap M_{sing}=\emptyset,$ therefore $K_{c}=V_{c}\cap K.$ Hence, it is enough to show that $V_{c}\cap K$ is polynomially convex.
		
		\medskip
		
		\noindent  {\bf Case I: $ \check{H}^{1}(K_{c},\mathbb{Z})=0.$} This implies $K_{c}$ is simply-coconnected (see \Cref{Rmk:Simply_Coconntd_Cech}). It is easy to see that $V_{c}$ is a Runge variety. Then by using \Cref{R:R_Convx_Runge_Varty}, we get that $K_{c}$ is rationally convex. Therefore, by using \Cref{R:R-conx_PolyConvx}, we can say that $K_{c}$ is polynomially convex in dimension one. Since $V_{c}$ is a one-dimensional Runge variety, therefore, $K_{c}\cap V_{c}=K_{c}$ is polynomially convex. 
		
		\medskip
		
		\noindent {\bf Case II: $ \check{H}^{1}(K_{c},\mathbb{Z})\ne 0$}. By the similar argument as \Cref{T:Polynomial} ({\bf Step I}, {\bf Case II}), $ \check{H}^{1}(K_{c},\mathbb{Z})\ne 0$ can not happen.
		\medskip
		
		\noindent Therefore, 	$K_{c}=P^{-1}\{c\}\cap K$ is polynomially convex.
		
		\medskip
		
		\noindent {\bf Step II:} {\bfseries\boldmath Completing the proof:}
		
		Now we show that $K$ is polynomially convex. Note that $h$ is holomorphic in $\cplx^2.$ We claim that $h(K)\cap h(\widehat{K}\setminus K)=\emptyset.$ If possible, assume that $h(\widehat{K}\setminus K)\not =\emptyset.$ Let $\alpha\in h(K)\cap h(\widehat{K}\setminus K).$ Then there exist $\eta\in K$ and $\xi\in \widehat{K}\setminus K$ such that $h(\eta)=\alpha$ and $h(\xi)=\alpha.$	
		We claim that  $\xi\in \widehat{h^{-1}\{\alpha\}\cap K}$ that is the fiber $\widehat{h^{-1}\{\alpha\}\cap K}$ is not polynomially convex. Let $Y=h(K).$ Since $h(K)\subset\partial\mathbb{D},$ each point of $Y$ is a peak point for the uniform algebra $\poly(Y),$ then by \Cref{R:Hull_fiber}, we obtain that
		
		\begin{align*}
			h^{-1}\{\alpha\}\cap\widehat{K}=\widehat{h^{-1}\{\alpha\}\cap K}.	
		\end{align*}
		Therefore,		
		
		\begin{align*}
			\xi\in h^{-1}\{\alpha\}\cap\widehat{K}=\widehat{h^{-1}\{\alpha\}\cap K}.	
		\end{align*}	
		Hence 		
		$\widehat{h^{-1}\{\alpha\}\cap K}$ is not polynomially convex. This is a contradiction. Hence $h(K)\cap h(\widehat{K}\setminus K)=\emptyset.$ Therefore, by \Cref{R:Hull_Img X}, $K$ is polynomially convex.		
	\end{proof}
	\begin{proof}[Proof of \Cref{Cor:SinLeviFlat_PolyCnvx}]
		We define the holomorphic function $h:\cplx^2\to \cplx$ by $h(z_1,z_2):=e^{z^{m}_1+z^{n}_2}.$ Then $\{(z_1,z_2)\in \cplx^2: |h(z_1,z_2)|=1\}=M$ and $M_{sing}=\{(0,0)\}.$ Now the proof follows from \Cref{T:Singular_Levi_flat}.
	\end{proof}
	We now give a proof of \Cref{T:Singular_Hypersurface}.
	\begin{proof}[Proof of Theorem~\ref{T:Singular_Hypersurface}]
		If $K$ is polynomially convex, then obviously $\widehat{K}\cap\{zw=0\}=\emptyset.$
		
		\medskip
		We now prove the converse part. Consider the domain $G:=\cplx^{2}\setminus \{(z,w)\in \cplx^2:zw=0\}$ and holomorphic function $P(z,w)=\frac{z}{w}.$ Then $M:=\rho^{-1}\{0\},$ where $\rho(z,w)=\frac{z}{w}\overline{\frac{z}{w}}-1.$ Note that since $dP|_{\alpha}\not =0,$ on $P^{-1}\{c\},$ then $\{z\in G:P(z)=c\}$ is a complex manifold of pure dimension $1.$ By \Cref{L:Charc_Foliation}, there exist characteristic foliation $K,$ and by \Cref{L:Func_Constant_Leaf}, we get that the function $h(z,w)=\frac{z}{w}$ is constant along each leaf of the characteristic foliation of $\triangle.$
		
		\medskip
		\noindent {\bf Step I:} {\bfseries\boldmath Showing that each fiber $P^{-1}\{c\}\cap K$ is polynomially convex.}
		
		\medskip	
		We define $K_{c}:=P^{-1}\{c\}\cap K$ and $V_{c}=\{(z,w)\in \cplx^2:z=cw\}.$ Since $K\cap \{(z,w)\in \cplx^2:zw=0\}=\emptyset,$ therefore $K_{c}=V_{c}\cap K.$ Hence, it is enough to show that $V_{c}\cap K$ is polynomially convex.
		
		\medskip
		
		\noindent  {\bf Case I: $ \check{H}^{1}(K_{c},\mathbb{Z})=0.$} This implies $K_{c}$ is simply-coconnected (see \Cref{Rmk:Simply_Coconntd_Cech}). Since $V_{c}$ is a Runge variety, by using \Cref{R:R_Convx_Runge_Varty}, we get that $K_{c}$ is rationally convex. Therefore, by using \Cref{R:R-conx_PolyConvx}, we can say that $K_{c}$ is polynomially convex in dimension one. Since $V_{c}$ is a one-dimensional Runge variety, therefore, $K_{c}\cap V_{c}=K_{c}$ is polynomially convex. 
		
		\medskip
		
		\noindent {\bf Case II: $ \check{H}^{1}(K_{c},\mathbb{Z})\ne 0$}. By the similar argument as \Cref{T:Polynomial} ({\bf Step I}, {\bf Case II}), $ \check{H}^{1}(K_{c},\mathbb{Z})\ne 0$ can not happen.
		\medskip
		
		\noindent Therefore, 	$K_{c}=P^{-1}\{c\}\cap K$ is polynomially convex.
		
		\medskip
		
		\noindent {\bf Step II:} {\bfseries\boldmath Completing the proof:}
		
		Now we show that $K$ is polynomially convex. By assumption, we have $\widehat{K}\cap \{(z,w)\in \cplx^{2}:zw=0\}=\emptyset,$ therefore, $P$ is holomorphic in a neighborhood of $\widehat{K}.$ We claim that $P(K)\cap P(\widehat{K}\setminus K)=\emptyset.$ If possible, assume that $P(\widehat{K}\setminus K)\not =\emptyset.$ Let $\alpha\in P(K)\cap p(\widehat{K}\setminus K).$ Then there exist $\eta\in K$ and $\xi\in \widehat{K}\setminus K$ such that $P(\eta)=\alpha$ and $P(\xi)=\alpha.$	
		We claim that  $\xi\in \widehat{p^{-1}\{\alpha\}\cap K}$ that is the fiber $\widehat{P^{-1}\{\alpha\}\cap K}$ is not polynomially convex. Let $Y=P(K).$ Since $P(K)\subset \partial\mathbb{D},$ each point of $Y$ is a peak point for the uniform algebra $\poly(Y),$ then by \Cref{R:Hull_fiber}, we obtain that
		
		\begin{align*}
			P^{-1}\{\alpha\}\cap\widehat{K}=\widehat{P^{-1}\{\alpha\}\cap K}.	
		\end{align*}
		Therefore,		
		
		\begin{align*}
			\xi\in P^{-1}\{\alpha\}\cap\widehat{K}=\widehat{P^{-1}\{\alpha\}\cap K}.	
		\end{align*}	
		Hence 		
		$\widehat{P^{-1}\{\alpha\}\cap K}$ is not polynomially convex. This is a contradiction. Hence $P(K)\cap P(\widehat{K}\setminus K)=\emptyset.$ Therefore, by \Cref{R:Hull_Img X}, $K$ is polynomially convex.		
	\end{proof}
	Let  $h:\cplx\to \mathbb{R}$ be a smooth function and $E_{h}:=\{z\in \cplx:h(z)=0\}.$ Let $E^{*}_{h}:=\{z\in G:h(z)=0\}$ be the regular part of $E_{h},$ where $G:=\cplx\setminus \{z\in \cplx:h(z)=0,dh(z)=0\}.$ By $(E_h)_{sing},$ we denote the singular part of $E_h.$
	\begin{theorem}\label{T:singular_cross-cgeneral}
		Let $K$ be a totally real disc in the singular hypersurface $E^{*}_{h}\times \cplx\subset \cplx^2.$ Then $K$ is polynomially convex.	
	\end{theorem}	
	
	\begin{proof}
		Proof follows from \Cref{T:Smooth_Cross_C}.
	\end{proof}
	
	\begin{remark}\label{T: normalform}
		Let $M:=\{z\in \cplx^2:\rho(z)=0\}$ be a singular Levi-flat quadratic real hypersurface in $\cplx^2$ with the following normal form (see \cite{DanielGong1999}):
		\begin{enumerate}[(i)]
			\item $\rho(z)=(z^2_1+2z_1\bar{z}_1+\bar{z}^2_1)$
			\item $\rho(z)=(z^2_1+2\lambda z_1\bar{z}_1+\bar{z}_1)~$ $\lambda\in (0,1)$
			\item $\rho(z)=\rl z_1\rl z_2$
		\end{enumerate}
		In the first case, $M$ is non singular. In the second case the singular set is $\{o\}\times\cplx$ and in the third case the singular set is $\{(iy_1,iy_2)\in\cplx^2:\;\; y_j\in\rea\}$. If a totally real disc lie in $M^*$, the non-singular part of $M$, then, by \Cref{T:singular_cross-cgeneral}, it is polynomially convex.
	\end{remark}

	\section{Examples and concluding remarks}
	In this section we provide some examples supporting the theorems in this paper. 
	\begin{example}\label{Exam:Ellipse}
		Let $E_{a,b}:=\{z\in \cplx:{a^2}{x^2}+{b^2}{y^2}=1\}$ $(a,b\in \mathbb{R}),$ and $M=E_{a,b}\times\mathbb{C}$ be a hypersurface $\mathbb{C}^2$. Then, by Theorem~\ref{T:Smooth_Cross_C}, every totally real smooth disk in $M$ is polynomially convex.
	\end{example}
	
	\begin{example}
		Let us consider $M:=\{(z_1,z_2)\in \cplx^2:|z_2+\phi(z_1)|=1\},$ $K:=\{(z_1,z_2)\in \cplx^2:z_2=e^{i(\rl z_1)^2}-\phi(z_1),|z_1|\le 1\}\subset M,$ where $\phi$ is a polynomial in $z_1.$ Using \Cref{T:Poly_Cnvx_OneDimntion}, we can show that $K$ is polynomially convex. Note that $K$ is not totally real (at $(0,1-\phi(0)).$ $K$ is the image of the closed unit disc under the smooth map $\xi\to (\xi,e^{i(\rl \xi)^2}-\phi(\xi)),$ which is essentially the graph of $e^{i(\rl z_1)^2}-\phi(z_1)$ over the closed unit disc. Hence, $K$ is simply-coconnected (in fact, contractible). It remains to show that each fiber $P^{-1}\{c\}\cap K$ is polynomially convex, where $P(z_1,z_2)=z_2+\phi(z_1)$. We compute:
		\begin{align*}
			P^{-1}\{c\}\cap K=&\{(z_1,z_2)\in \cplx^2:z_2+\phi(z_1)=c, w=e^{i(\rl z_1)^2}-\phi(z_1),|z_1|\le 1\}\\
			=& \{(z_1,z_2)\in \cplx^2:z_2+\phi(z_1)=c, e^{i(\rl z_1)^2}=c=e^{it},|z_1|\le 1, t\in [0,2\pi]\}\\
			=& \{(z_1,z_2)\in \cplx^2:z_2+\phi(z_1)=c, \rl z_1=\pm t,|z_1|\le 1\}\\
			=& \Gr_{h}(L_{1}\cup L_{2}),
		\end{align*}
		where $h=c-\phi(z_1),$ $L_{1}\cup L_{2}$ is polynomially convex. Therefore, being a graph of a polynomial over polynomially convex set, $P^{-1}\{c\}\cap K$ is polynomially convex.
	\end{example}

	\begin{example}
		Let $P(z_1,z_2)=\frac{1}{2}(z_1+z_2)$ and $M:=\{(z_1,z_2)\in \cplx^2:|P(z_1,z_2)|=1\}.$ Let us consider $\phi:\overline{\mathbb{D}}\to \cplx^2$ be a map defined by $\phi(t,s)=(\cos{t}+t,\sin{s}+s,\cos{t}-t,\sin{s}-s).$ Then $\phi$ is a diffeomorphism and hence $K:=\phi(\overline{\mathbb{D}})$ is a simply-coconnected compact subset of $M.$ Note that $P^{-1}\{c\}\cap K=\{(c+\overline{\xi},c-\overline{\xi}):|\xi|\le 1\}$ is polynomially convex. Using \Cref{T:Poly_Cnvx_OneDimntion}, we conclude that $K$ is polynomially convex.
	\end{example}

	\begin{example}
		Let $h:\cplx\to \mathbb{R}$ be defined by $h(x,y)=x^2+y^2-1$ and $\phi:\overline{\mathbb{D}}\to \mathbb{C}^2$ be defined by $\phi(t,s)=(\cos{t},\sin{t},(t-3)^3,(s-2)^3).$ Then $\phi$ is a diffeomorphism. Hence $K:=\phi(\overline{\mathbb{D}})$ is a simply-coconnected compact subset of $E_{h}\times \cplx.$ Let $P(z,w)=z$ be the projection map. Then $P^{-1}\{c\}\cap K=\{(z,w): z=c,w=((t-3)^3,(s-2)^3)), t^2+s^2\le 1\}$ is polynomially convex for all $c\in \cplx.$ Using \Cref{T:smooth_cross-cgeneral}, we conclude that $K$ is polynomially convex.
	\end{example}
	
We now make some concluding remarks. The hypersurfaces considered in this paper are all globally Levi-flat, i.e. pseudoconvex from both sides. In all the non-singular hypersurfaces considered here, any totally real disc lying in them turned out to be polynomially convex. This encourages us to make our first conjecture: 
\smallskip

\noindent
\begin{conjecture}
	Any totally real disc lying in a non-singular Levi-flat hypersurface in $\cplx^2$ is polynomially convex.
\end{conjecture}

In view of \Cref{Cor:SinLeviFlat_PolyCnvx}, \Cref{T:Singular_Hypersurface} and \Cref{T: normalform}
we infer that the totally real discs lying in the Levi-flat quadrics, which are the normal form given by Burns and Gong \cite{DanielGong1999}, are polynomially convex if  
the polynomial hull of a totally real disc does not have an intersection with the singular set. This allows us to make the next conjecture. 
\begin{conjecture}
	Any totally real disc that $K$ lying in a singular Levi-flat hypersurface is polynomially convex if $\hull{K}$ does not have a non-empty intersection with the singular set.
\end{conjecture}
\noindent In the third hypersurface of \Cref{T: normalform} the singular set is of the form 
$S:=\{(iy_1,iy_2)\in\cplx^2:\;\; y_j\in\rea\}$, which is a totally real subspace of maximal dimension in $\cplx^2$. Hence, any compact subset lying in $S$ is polynomially convex. A totally real disc can lie there. Hence, we ask the following question:
\begin{question}
	Characterize the totally real discs in singular Levi-flat hypersurfaces in $\cplx^2$ which are polynomially convex.
\end{question}

\medskip

\noindent {\bf Acknowledgements.} The first named author is supported partially by a MATRICS Research Grant (MTR/2017/000974) and by a Core Research Grant (CRG/2022/003560) of SERB, Dept. of Science and Technology, Govt. of India. The work of the second named author is partially supported by an INSPIRE
Fellowship (IF 160487), Dept. of Science and Technology, Govt. of India, and by a research grant of SERB (30121582), Dept. of Science and Technology, Govt. of India.

	

\end{document}